\documentclass[oneside,reqno,12pt]{amsart}
\usepackage[left=1in,top=1in,right=1in,bottom=1in]{geometry}
\usepackage{amsmath,amsfonts,amssymb,amsthm}
\usepackage{verbatim}
\usepackage{latexsym}
\usepackage[colorlinks,citecolor=red,pagebackref,hypertexnames=false]{hyperref}

\newtheorem{theorem}{Theorem}[section]
\newtheorem{lemma}[theorem]{Lemma}

\newtheorem{definition}[theorem]{Definition}
\newtheorem{corollary}[theorem]{Corollary}

\newtheorem{example}[theorem]{Example}

\newtheorem{claim}[theorem]{Claim}

\newcommand\RR{{\Bbb R}}
\newcommand\CC{{\Bbb C}}

\newcommand\NN{{\Bbb N}}
\newcommand\ZZ{{\Bbb Z}}

 \newcommand\HH{{\Bbb H}}

\newcommand{\note}[2][\null]{%
  \marginpar{\renewcommand{\baselinestretch}{1}\vspace{-1em}\hrule\vspace{3pt}%
  \tiny\raggedright{#2\ifx#1\null\else\\\hfill---
  {\em #1}\fi}\vspace{1.5em}}
  }

\begin{document}
\title{Identification of  Besov  spaces via  Littlewood-Paley-Stein Type  $g$-functions}

\author{Azita Mayeli \\
\address{City University of New York, 
  {\tt amayeli@gc.cuny.edu} 
 }}
\keywords{Inhomogeneous Besov space,  Littlewood-Paley-Stein type $g$-functions, square functions, symmetric diffusion semigroups, bandlimited functions\\  The research was supported by PSC-CUNY (The Professional Staff Congress - City University of New York) 
Grant, Tradition B
}
\subjclass[2010]{43, 46 (primary), and 41  (secondary)}
\date{\today}

   \begin{abstract}
We use   Littlewood-Paley-Stein type $g$-functions  (also called generalized square functions)    associated to  symmetric diffusion semigroups     to obtain a characterization of   inhomogeneous abstract Besov spaces on the abstract Hilbert spaces.  Then we apply our results  for the abstract Besov spaces defined through  the Poisson and Gauss-Weierstrass semigroups. 
      \end{abstract}

  \maketitle
  
 \section{Introduction}\label{introduction}
 

Let $\{T_t\}_{t\geq 0}$ be  a symmetric diffusion semigroup on a Hilbert space $\mathcal H$.  (See  Section \ref{notations}  for the definition.). 
For $1\leq q<\infty$, $\alpha>0$, 
the   {\it inhomogenous abstract Besov space} $B_q^\alpha: = B_q^\alpha(\mathcal H)$ is the space of  all  functions $f\in \mathcal H$ for which
  \begin{align}\notag
   \int_0^1 \left(s^{-\alpha}  \Omega_{r}(s,\>f )\right)^q   \ \frac{ds}{s} <\infty. 
  \end{align}

\noindent Here, $\Omega_{r}(s,\>f )$ is 
 the  {\it  modulus of continuity}  for    $r\in \NN$ with $r\geq \alpha$ and is  defined by 
\begin{equation}\notag
\Omega_{r}(s,\>f )=\sup_{0<\tau\leq s} \| \left( I-T_\tau \right)^r f\|~.
\end{equation}

 \noindent 
  The space $B_q^\alpha$, $1\leq q<\infty$, is a  Banach space with the norm

  \begin{align}\label{Besov-norm-via-modulus}
 \|f\|_{B_q^\alpha} =  \|f\| + \left( \int_0^1 \left(s^{-\alpha}  \Omega_{r}(s,\>f )\right)^q   \ \frac{ds}{s}\right)^{1/q}.
  \end{align}
  
 We use the standard convention for the definition of the norm when $q=\infty$. 
  Notice that the Besov norm is independent of the choice  of  $r$ due to the  monotonicity of the modulus of continuity $\Omega_r$. Therefore, for fixed $q$ and $\alpha$,  the   definition  of Besov space  is independent of the choice of $r$.  
  \\

In  our previous paper \cite{Mayeli-PAMS},  we proved    the norm in   (\ref{Besov-norm-via-modulus})  can be  expressed  in terms of ``smooth\rq\rq{} and ``band-limited\rq\rq{} functions using the  Littlewood-Paley decomposition. 
In this  paper  we show  this norm  is also equivalent to a norm which can be expressed  in terms of  {Littlewood-Paley-Stein type $g$-functions}  as  we will explain.    \\ 

  For  any given  $m\in \NN$, $1\leq q<\infty$,   $\beta\in\RR$ and  a diffusion semigroup $\{T_t\}_{t\geq 0}$    on a Hilbert space  $\mathcal H$,  
  we denote a  bounded map from $\mathcal H$ to $\mathcal H$ 
   by  $\mathcal G_{m,q,\beta}$ 
  and call it  a  {\it Littlewood-Paley-Stein type $g$-function}  if the following property holds:

 \begin{align}\label{$g$-function}
 \|\mathcal G_{m,q,\beta}(f)\|^q\asymp
 \int_{0}^\infty t^{(m-\beta)q} \left\|\cfrac{\partial^m}{\partial t^m}T_tf\right\|^q \ \frac{dt}{t} . 
\end{align}

 
   

  
\noindent  For $q=\infty$ we use the standard convention.   \\ 


The functions   $\mathcal G_{m,q,\beta}$ can be interpreted as  a generalized version of   {\it $g$-functions} or {\it square functions} in abstract form.  
 In the classical Littlewood-Paley theory, various $g$-functions   have played a central role.   Typical examples  are 
 \begin{align}\label{exmpl1}
 g_k(f)(x)= \left(\int_0^\infty |t^k\cfrac{\partial^k}{\partial t^k} T_t f(x)|^2 \frac{dt}{t}\right)^{1/2}
 \end{align}
and 
 \begin{align}\label{exmpl2}
 g_k(f)(x)= \left(\int_0^\infty |(tA)^k\cfrac{\partial^k}{\partial t^k} T_t f(x)|^2 \frac{dt}{t}\right)^{1/2}, 
 \end{align}
  where $f\in L^p$, $A$ is the  infinitesimal operator of the semigroup and $k\in \NN$. (Note, for $p=2$, if we take the $L^2$-norm of (\ref{exmpl1}), we arrive at \ref{$g$-function}) with $q=2$, $\beta=0$ and $m=k$. Therefore, our definition of a Littlewood-Paley-Stein type $g$-function on an abstract Hilbert space  as a generalized  version of a $g$-function makes intuitive sense.)  \\

The $g$-functions for the  generators of semigroups grew out of classical harmonic analysis and the Littlewood-Paley theory. They were first developed in E. M. Stein\rq{}s classical book 
 \cite{Stein70}. 
Also  see the monograph \cite{Frazier-Jawerth-Weiss91} for  when $k=1$ and    \cite{Meda95} for  any $k\in \NN$.  As it has been shown in  \cite{Frazier-Jawerth-Weiss91},  the $g$-functions associated to the Possion semigroup  can be  applied to determine  the bounded solutions of a Dirichlet problem on $\RR^n\times \RR^+$ with  bounded boundary value function. 
Note   in his  work \cite{Meda95}, the author proves the $L^p$-boundedness of $g$-functions.  These functions  here are associated to $(tA)^s T_t$, where $s>0$ and  $A$ is the infinitesimal operator for the semigroup $\{T_t\}_{t\geq 0}$.   
 Another principal motivation for introducing the $g$-functions  is   to provide new equivalent norms for
the spaces $L^p$, which make the boundedness properties of various operators derived from the
semigroup $T_t$ more transparent. For further applications of this form,  for example  in unconditionally martingale difference spaces,  we refer  the reader to  \cite{hytoenen}, where  $\mathcal H=\RR^n$, and $\beta=0$.  \\

The main result of this paper,  which characterizes   the inhomogenous abstract Besov spaces in terms of the  Littelwood-Paley-Stein type  $g$-functions,  reads as follows. 

\begin{theorem}\label{Stein-LPW-decomp} Let $\{T_t\}_{t\geq 0}$ be a symmetric diffusion semigroup. Then for  $1<q<\infty$,  $\alpha>1/2$, and $m\in \NN$ with  $m>2\alpha$,   there is    $\beta:=\beta(\alpha)\in\Bbb R$ such that

 \begin{align}\label{Stein-LPW-Besov-norm} 
 \|f\|_{B_q^\alpha} \asymp \|f\| +   \|\mathcal G_{m,q,\beta}(f)\|, 
 \end{align}
 or equivalently,  

 \begin{align}\label{Stein-LPW-Besov-norm2}
 \|f\|_{B_q^\alpha}  \asymp   \|f\| +  \left(\int_0^\infty t^{(m-\beta)q} \left\|\cfrac{\partial^m}{\partial t^m}T_tf\right\|^q \  \frac{dt}{t}\right)^{1/q} 
 \end{align}
 with  $m>\beta$. 
 The result holds for $q=\infty$ with  the modification of the  definition. 
\end{theorem}

The definition of Besov spaces in terms of semigroups and modulus of continuity for abstract Hilbert spaces was introduced, for example,  by  Lion  (\cite{Lion}).  
   There is a large amount of work in the study of Besov spaces and their characterizations in terms of Littlewood-Paley $g$-functions   
 for when $\mathcal H$ in $B_q^\alpha(\mathcal H)$ is replaced by $L^p(X)$ space. 
   For example, see  \cite{Grigor-Liu-heat semi,Bui-Doung-Yan12,Trieb1}  for  a characterization of $B_q^\alpha(L^p)$  
   in terms  of the Weierstrass or  heat semigroup $T_t=e^{-Lt}$ (convolution with the heat kernel) associated to  a self-adjoint positive definite operator  $L$ on $L^2$.\\  



The organization of this paper is as follows: After some historical comments in Section \ref{historical comments}, in Section \ref{notations} we introduce some notations and definitions, including a short description of   our previous results in  \cite{Mayeli-PAMS}.  
  In Section \ref{main results} we  prove few key lemmas and the main result in   Theorem \ref{Stein-LPW-decomp}.  Finally we apply our results to illustrate Littlewood-Paley-Stein type $g$-functions   associated to  the Poisson  and Gauss-Weierstrass semigroups.

  
\section{ Historical comments on  Besov spaces}\label{historical comments}

In the classical setting,  Besov spaces  $B_q^\alpha(L^p)$  are the set of all functions in $L^p$ with smoothness degree $\alpha$ where their  (quasi)norm is  controlled by $q$. These spaces appear in many subfields of analysis and applied mathematics and have two types.   The definition of  one type uses the Fourier transform (for example see \cite{Peetre,Tr3}), while the other uses the modulus of continuity or smoothness. The spaces  defined by smoothness are  more practical in many areas of applied analysis, such as in approximation theory and the decomposition of signals  (\cite{D-P,Donoho-Johnston1995,DeVore-Popov88}). \\

For application  purposes,   
 it is natural to decompose a Besov function  into simple building blocks and  as a result to reduce the study of functions to the study of only  the elements in the decomposition.   Wavelet and frame theory have played a key role to achieve this goal.   In the classical level, this kind of decomposition in terms of \lq\lq{}smooth  wavelets\rq\rq{}
using spectral theoretic approach was proved in  \cite{FJ}. 
  A unified characterizations of Besov spaces in terms of atomic decomposition using a  group representation theoretic approach was given
by Feichtinger and Gr\"ochenig (\cite{FG}).  
New results in this direction in the context of Lie groups and homogeneous manifolds  were recently published in \cite{CO1,CMO1,Fueh,FM1}, and \cite{gm1}-\cite{gm4}. For the  classification of Besov spaces on compact Riemannian manifolds using continuous and time-frequency localized wavelets with higher vanishing moments we invite the reader to see \cite{gm2,gm3}. 
  For  other equivalent definitions of these spaces in terms of abstract wavelets and  as interpolation spaces between Hilbert spaces and Sobolev spaces, for example,   see the papers \cite{KPS,Tr3}.   \\

     As we mentioned earlier, 
 in this paper we present a description of $B_q^\alpha(\mathcal H)$
 in term of   $g$-functions of Littlewood-Paley-Stein type. 
 While the idea behind such an identification is simple, the  proofs are very technical and the  main difficulties arise when we replace $L^2$    by any   Hilbert space $\mathcal H$. 

\section{Preliminaries and Definitions}\label{notations}
 
Let $\mathcal H$ be a Hilbert space and  $\{T_t\}_{t>0}$ be a semigroup on $\mathcal H$. We set    $T_0=I$.   For any $f\in \mathcal H$, all $T_t f$ belong to $\mathcal H$  and 
 $\lim_{t\to 0^+} T_tf = f$ (\cite{Rudin}). 
We say $\{T_t\}_{t\geq 0}$ is  {\it symmetric diffusion} if it satisfies  the following conditions. 
  \begin{enumerate}
  \item   $T_t$ are contractions on $\mathcal H$, i.e.,  $\|T_tf\|\leq  \|f\|$ for all $f\in \mathcal H$. 
  \item  $T_t$ are symmetric, i.e., each $T_t$ is self-adjoint on $\mathcal H$. 
   \item $T_t$ are positivity preserving, i.e., 
   $T_tf\geq 0$ if $f\geq 0$.  
\end{enumerate}


 \noindent Symmetric diffusion semigroups occur often in analysis. For  examples of this type    see \cite{Stein70}. \\

Let $A$ be the infinitesimal generator of the semigroup $\{T_t\}_{t\geq 0}$ with domain $\mathcal D(A)$. By  the definition of   Besov spaces, the following inclusions hold (\cite{Berens-Butzer64}):
 \begin{align}\notag
 \mathcal D(A)\subseteq B_q^\alpha\subseteq \mathcal H . 
 \end{align} 

 \noindent $A$  is a     densely  defined closed linear  operator on $\mathcal H$ such that 
  \begin{align}\label{limit}
  \lim_{t\to 0^+} \left\|\cfrac{T_tf -f}{t} - Af\right\|= 0 ,   \quad \quad  \forall \ f\in\mathcal H .
  \end{align}

 
  \noindent By the functional calculus,  
  $T_tf=e^{tA}f$  for all  $f\in\mathcal H$ and   $A$ has a representation
\begin{align}\label{spectral representation}
A= \int_0^\infty \lambda d\mu_{\lambda} ~,
\end{align}
where $d\mu_{\lambda}$ is a projection valued measure (\cite{ReedSimon}). This implies that for any  $f\in \mathcal D(A)$  and $g\in  \mathcal H$  
 \begin{align}\label{inner product} 
 \langle A f, g\rangle = \int_0^\infty \lambda~  d(\mu_\lambda f, g). 
 \end{align}
The inner product  induces an equivalent definition for the 
  domain of  $A$  given by 
  $$\mathcal D(A)=\left\{ f\in  \mathcal H : ~~~ \|A f\|^2:=\int_0^\infty \lambda^2 d(\mu_\lambda f, f)<\infty\right\}.$$
   By the density, the inner product (\ref{inner product}) also holds for all $f\in \mathcal H$. 
 As a result of   (\ref{limit}) and  (\ref{spectral representation}), 
   %
   %
%
the operator  $T_t$   has a spectral decomposition 
 $$T_t= \int_0^\infty e^{-\lambda t} d\mu_\lambda,  \quad \quad  t>0 .   $$

\noindent Thus 

 $$\cfrac{\partial^m}{\partial t^m}T_t= (-1)^m\int_0^\infty \lambda^m e^{-\lambda t} d\mu_\lambda ~ , \quad  \quad\forall \ \  m\in \NN. $$



We say $f\in\mathcal H$ is a {\it  Paley-Wiener}  or {\it bandlimited} function with respect to the operator $A$ and the projection valued measure $\mu_\lambda$ if $d(\mu_\lambda f, f)$ is supported in the interval $[a, b]$, $0<a<b<\infty$, i.e., 
 \begin{align}\notag
 \langle A f, f\rangle  = \int_a^b \lambda d(\mu_\lambda f, f) .
 \end{align} 

We denote the space of such functions by $PW_{[a,b]}(A)$ and call  it   {\it Paley-Wiener space}.   (For an 
 equivalent definition of these spaces 
  using the so-called functional form of the spectral theorem see, for example, \cite{BS}.)
 A vector  $f$ in $\mathcal H$ is  smooth if  it belongs to    $ \mathcal D(A^k)$,     $\forall \  k\in \Bbb N$. It is straightforward   that
  $PW_{[a,b]}(A)  \subseteq \cap_{k\in\NN} \mathcal D(A^k)$, so every vector in  the Paley-Wiener space  is smooth.  \\
%


Let  $\{\psi_j\}_{j\geq 0}$ be a sequence of  bounded real-valued functions on $[0,\infty)$ with  ${\text supp}(\hat \psi_j) \subseteq [2^{j-1}, 2^{j+1}]$. Assume that   the following  
 resolution of identity  (a.k.a. a discrete version of Calder\'on decomposition)  holds:

\begin{align}\label{resolution of identity}
 \sum_{j=0}^\infty \hat\psi_j(\lambda)^2= 1 .
\end{align} 
For any $f\in \mathcal H$ and $ j\geq 0$,  
define   $f_j:= \hat\psi_j(A)f$. 
Therefore    by the resolution of identity and the functional calculus: 

\begin{align}\label{decomp}
f= \sum_{j=0}^\infty  \hat\psi_j(A) f_j \ .
\end{align}
Since  $\hat\psi_j(A)f\in \cap_{k\in\NN} \mathcal D(A^k)$, the
 functions  $f_j$ are  band-limited 
and  smooth. 
 Therefore (\ref{decomp}) represents a decomposition of $f$ in terms of band-limited and smooth functions. 
 In \cite{Mayeli-PAMS} we applied  (\ref{decomp})  to prove  that for 
$\alpha>1/2$ and  $1<
q\leq \infty$ 
 
\begin{equation}
\|f\|_{B_q^\alpha} \asymp  \|f\|+ \left(\sum_{j=0}^{\infty}\left(2^{j\alpha
}\|\hat\psi_j(A)f\|\right)^{q}\right)^{1/q} , \quad  \forall \ f\in \mathcal H,
\label{normequiv}
\end{equation}
 provided that  both sides  are finite. As usual,   we use   the standard modifications for $q=\infty$.   We will use the decomposition (\ref{decomp})  to  prove  the main  result in  Theorem \ref{Stein-LPW-decomp}  in this paper. \\ 
 
  %



 {\it Notation:} By  $\| \ \|_{op}$ we shall mean  the  operator norm, and we will use 
   $\preceq (\succeq)$    when the inequality    $\leq (\geq)$ holds  up  to some uniform constant. Throughout the paper, the equivalence  $\asymp$  indicates $\preceq$ and $\succeq$. By the Besov space (norm)  we shall  mean  the  abstract inhomogeneous Besov space  (norm).

\section{Proof of Theorem  \ref{Stein-LPW-decomp}}\label{main results}

 We need some key lemmas  here before we prove our main result.
  The first one follows. 
 
\begin{lemma}\label{fundamental estimation} For any $j\geq 0$, $1\leq p<\infty$ and $k, m\in \Bbb N$   
 
\begin{align}\label{xx}
\int_0^\infty  t^{pk}
 \ \left\| \cfrac{\partial^m T_{t}}{\partial t^m}\hat\psi_j(A)\right\|_{op}^p
\frac{dt}{t} \preceq \ 2^{-jp(k-m)} ~.
\end{align} 
\end{lemma}
\begin{proof}
By the  functional calculus we have

 $$\left\|\cfrac{\partial^m T_{t}}{\partial t^m}\hat\psi_j(A)\right\|_{op} = \sup_{\lambda>0} \left| (-\lambda)^me^{-\lambda t} \hat\psi_j(\lambda)\right| = \sup_{2^{j-1}\leq \lambda\leq 2^{j+1}} |(-\lambda)^m e^{-\lambda t} \hat\psi_j(\lambda)|\leq  c_1 \ 2^{mj}  e^{-2^{j-1}t}.   $$
 
\noindent  By applying this in (\ref{xx}): 

  $$ \int_0^\infty  t^{pk}
 \ \left\| \cfrac{\partial^m T_{t}}{\partial t^m}\hat\psi_j(A)\right\|_{op}^p
\frac{dt}{t}  \preceq 2^{jmp} \int_0^\infty  t^{pk} e^{-2^{j-1}pt} \frac{dt}{t} 
=c  2^{-jp(k-m)} ~ .
   $$
Here,  the  constants $c_1$  and $c$ are independent of  $j$ and $t$.
  This completes the proof of the lemma. 
  \end{proof}

 \begin{lemma}\label{Kf}  Let $M$and $m$ be positive integers  such that    
 $M>2m$. 
  Let $\beta\in \RR$. 
  Then the operator    $K$ defined on  $\mathcal H$ by  
$$Kf= \int_0^\infty   t^{M-\beta} \left(\cfrac{\partial^m T_{t}}{\partial t^m}\right)^2f \ \frac{dt}{t} $$
 is bounded and  for $1<q\leq \infty$ 
$$ \| Kf\| \preceq  \|\mathcal G_{m,q,\beta}(f)\| .$$
 
  \end{lemma} 
   \begin{proof}  By the  partition of unity  in 
(\ref{resolution of identity}), 
for any $f\in \mathcal H$  we have 

\begin{align}\label{commutative-relation} 
Kf= \sum_{j\geq 0} \hat\psi_j^2(A) Kf= \sum_{j\geq 0}   \hat\psi_j(A)Kf_j , 
\end{align}

\noindent where $f_j= \hat \psi_j(A)f$. 
Notice that above we used the fact that  $K$ commutes with  $\hat\psi_j(A)$. 
   We continue as follows: First let $1<q<\infty$. Then 
 
\begin{align}\notag
\|K f_j\| & \leq \int_0^\infty  t^{M-\beta}  \left\|
\left(\cfrac{\partial^m T_{t}}{\partial t^m}\right)^2 \hat\psi_j(A)f\right\| \cfrac{dt}{t} \\\label{eq1}
&\leq  \int_0^\infty  t^{M-\beta} \left\|\cfrac{\partial^m T_{t}}{\partial t^m}f\right\|  \ \left\| \cfrac{\partial^m T_{t}}{\partial t^m}\hat\psi_j(A)\right\|_{op}\cfrac{dt}{t} \\\label{eq2}
& \leq \left(\int_0^\infty  t^{(M-\beta)q} \left\|\cfrac{\partial^m T_{t}}{\partial t^m}f\right\|^q  \cfrac{dt}{t}   \right)^{1/q}
 \left(\int_0^\infty  t^{(M-m)p}
 \ \left\| \cfrac{\partial^m T_{t}}{\partial t^m}\hat\psi_j(A)\right\|_{op}^p
  \cfrac{dt}{t}\right)^{1/p}\\\label{upper-bound}
 & \preceq  \|\mathcal G_{m,q,\beta}(f)\| \left(\int_0^\infty  t^{(M-m)p}
 \ \left\| \cfrac{\partial^m T_{t}}{\partial t^m}\hat\psi_j(A)\right\|_{op}^p
\cfrac{dt}{t}\right)^{1/p}  
  \end{align}
\noindent  with $\cfrac{1}{p}+ \cfrac{1}{q} =1$. To pass from (\ref{eq1}) to (\ref{eq2}) we used  H\"older\rq{}s inequality. 
The inequality (\ref{upper-bound}) holds by the definition of Littlewood-Paley-Stein type $g$-functions.  By 
     Lemma \ref{fundamental estimation}, for $k=M-m$ we can write 
  
\begin{align}\label{second-int}
\int_0^\infty  t^{(M-m)p}
 \ \left\| \cfrac{\partial^m T_{t}}{\partial t^m}\hat\psi_j(A)\right\|_{op}^p
\cfrac{dt}{t}  \preceq 2^{-j(M-2m)p} , 
 \end{align}
where the constant in the inequality  is independent of $j$.  
By applying  this estimation for  (\ref{upper-bound}) we arrive at

\begin{align}\label{greater side}
\|K f_j\|  \preceq 2^{-j(M-2m)} \|\mathcal G_{m,q,\beta}(f)\|.
\end{align}

 \noindent We will use these estimations for $Kf_j$  in   (\ref{commutative-relation}) to complete the proof as follows.

\begin{align}\notag
\| Kf\| &= \|\sum_{j\geq 0} \hat\psi_j^2(A) Kf\| \\\label{call something} 
  &\leq \sum_j \|Kf_j\|  \| \hat\psi_j(A)\|_{op}. 
  \end{align} 
  
\noindent  Due to (\ref{resolution of identity}) we have  $\|\hat\psi_j(A)\|_{op}\leq 1$. Therefore, with (\ref{greater side}) 
  \begin{align} 
  (\ref{call something}) 
&\preceq  \|\mathcal G_{m,q,\beta}(f)\|  \sum_{j\geq 0} 2^{-j(M-2m)}  .
\end{align}
 
\noindent With the assumption  $M>2m$, the  sum is finite. Thus
\begin{align}
\| Kf\|  \preceq  \|\mathcal G_{m,q,\beta}\| , 
\end{align}
as desired. For $q=\infty$, we use the standard convention for the definition of the norm. 
%
 \end{proof}

 %

\begin{lemma}\label{an estimation for inner product}
Let $m, j$ be as in the above and $t, s > 0$. Then 

$$\left\langle\left(\cfrac{\partial^m T_{t}}{\partial t^m}\right)^2f_j, \left(\cfrac{\partial^m T_{s}}{\partial s^m}\right)^2f_j\right\rangle    \geq 2^{4m(j-1)} e^{-2^{j+2}(t+s)}\|f_j\|^2  .$$
\end{lemma} 

\begin{proof} 
By the spectral theory and  the assumption on the support of $\hat \psi_j$ 
  we have 
  
\begin{align}\notag
\left\langle\left(\cfrac{\partial^m T_{t}}{\partial t^m}\right)^2f_j, \left(\cfrac{\partial^m T_{s}}{\partial s^m}\right)^2f_j\right\rangle&=   \left\langle \int_{\lambda>0}  \lambda^{2m} e^{-2t\lambda} d\mu_\lambda f_j, \int_{w>0}  w^{2m}e^{-2s w} d\mu_w f_j\right\rangle \\\notag
  &= \left\langle \int_{\lambda>0}  \lambda^{2m}  e^{-2t\lambda}\hat\psi_j(\lambda) d\mu_\lambda f, \int_{w>0}   w^{2m}  e^{-2s w} \hat\psi_j(w) d\mu_w f\right\rangle  \\\notag
& =
\left\langle \int_{\lambda=2^{j-1}}^{2^{j+1}}  \lambda^{2m} e^{-2t\lambda} \hat\psi_j(\lambda) d\mu_\lambda f , \int_{w=2^{j-1}}^{2^{j+1}}    w^{2m}  e^{-2s w} \hat\psi_j(w)     d\mu_w f\right\rangle\\\notag
 &=  \int_{\lambda=2^{j-1}}^{2^{j+1}}  \lambda^{2m}   e^{-2 t\lambda }\  d \left(
 \int_{w=2^{j-1}}^{2^{j+1}}  w^{2m}  e^{-2sw} d\langle \mu_\lambda f_j, \mu_w f_j\rangle  \right)\\\label{estimation-x}
 &\geq 2^{4m(j-1)}
  e^{-2^{j+2}(t+s) }\int_{\lambda=2^{j-1}}^{2^{j+1}}  d\left(
 \int_{w=2^{j-1}}^{2^{j+1}}     d\langle \mu_\lambda f_j, \mu_w f_j\rangle  \right) .
 \end{align}


 Next we aim to show that the integral on the right is $\|f_j\|^2$: 
Notice by the decomposition (\ref{spectral representation}) and the spectral theory we can write 
 \begin{align}\label{decomp for psi(A)} 
 \hat\psi_j(A) f=   \int_{\lambda=0}^\infty  \hat \psi_j(\lambda) d\mu_\lambda f = 
\int_{\lambda=2^{j-1}}^{2^{j+1}} \hat \psi_j(\lambda) d\mu_\lambda f = 
\int_{\lambda=2^{j-1}}^{2^{j+1}}   d\mu_\lambda f_j 
  \quad \quad  \forall \ f\in \mathcal H. 
  \end{align}  

\noindent Therefore, for any $g\in \mathcal H$

$$\langle \hat\psi_j(A) f , g\rangle  = \int_{\lambda=2^{j-1}}^{2^{j+1}}   d\langle \mu_\lambda f_j, g\rangle .$$ 

\noindent  Take $g=f_j=\hat\psi_j(A) f$. The preceding equation with the decomposition (\ref{decomp for psi(A)}) implies that

\begin{align}\label{inner-product}
 \int_{\lambda=2^{j-1}}^{2^{j+1}}  d\left(
 \int_{w=2^{j-1}}^{2^{j+1}}     d\langle \mu_\lambda f_j, \mu_w f_j\rangle  \right) = 
 \|f_j\|^2 .
 \end{align} 
  

\noindent  With interfering  (\ref{inner-product})   in (\ref{estimation-x}), the proof holds: 

 \begin{align}\notag
\left\langle\left(\cfrac{\partial^m T_{t}}{\partial t^m}\right)^2f_j, \left(\cfrac{\partial^m T_{s}}{\partial s^m}\right)^2f_j\right\rangle  \geq 2^{4m(j-1)} e^{-2^{j+2}(t+s)}\|f_j\|^2  .
\end{align}
\end{proof} 

  \begin{lemma}\label{lower bound theorem}  For $M>2m$, $M>2q\alpha-\beta$ and  $f\in \mathcal H$ 
  
 \begin{align}\label{something} 
 \|f_j\|   \preceq 2^{-j(4m-M+\beta)/2}  \|K f_j\| ,  
 \end{align} 
   where $f_j = \hat\psi(A)f$, $j\geq 0$, and  $K$ is the operator defined in Lemma \ref{Kf}.
   
  \end{lemma}

  \begin{proof}



By the definition of the operator $K$ we have 

\begin{align}\label{lower-bound}
\|K f_j\|^2  &=  \int_{t=0}^\infty \int_{s=0}^\infty  (ts)^{M-\beta}  \left\langle\left(\cfrac{\partial^m T_{t}}{\partial t^m}\right)^2f_j, \left(\cfrac{\partial^m T_{s}}{\partial s^m}\right)^2f_j\right\rangle \frac{dt}{t} \frac{ds}{s}.   
\end{align}

\noindent Then  (\ref{something}) is an immediate result  of  Lemma \ref{an estimation for inner product} and the following calculations. 

\begin{align}\notag
\|Kf_j\|^2 &     
  \geq  \int_{t=0}^\infty \int_{s=0}^\infty  (ts)^{M-\beta}   2^{4m(j-1)} e^{-2^{j+2}(t+s)}\|f_j\|^2
\frac{dt}{t} \frac{ds}{s} \\\notag
& = 2^{4m(j-1)}   \|f_j\|^2  \int_{t=0}^\infty \int_{s=0}^\infty    (ts)^{M-\beta}  e^{-2^{j+2}(t+s)}
 \frac{dt}{t} \frac{ds}{s} \\\notag
&= 2^{4m(j-1)} \|f_j\|^2 \left(\int_{t=0}^\infty  t^{M-\beta}   e^{-2^{j+2}t}
\frac{dt}{t}\right)^2 \\\notag
&= 2^{4m(j-1)} 2^{-(j+2)(M-\beta)} \|f_j\|^2 \left(\int_{t=0}^\infty  t^{M-\beta}  e^{-t}
\frac{dt}{t} \right)^2 \\\notag
& = c  \ 2^{j(4m-M+\beta)} \|f_j\|^2 , 
\end{align}
where $c= 2^{-2(2m + M-\beta)}$.  This completes the proof of the  lemma.

\end{proof}

\begin{proof}[{\bf Proof of Theorem \ref{Stein-LPW-decomp}}] 
 
 First we show that for $1<q<\infty$ 
 
  $$\|f\|_{B_{\alpha}^q} \preceq \|f\|+  \left(\int_0^\infty  t^{(m-\beta)q} \left\|\cfrac{\partial^m T_{t}}{\partial t^m}f_j\right\|^q   \frac{dt}{t}\right)^{1/q} .
 $$
 
 \noindent As a result of  Lemma \ref{lower bound theorem} and the inequality   (\ref{greater side})   we have 

 \begin{align}\notag
 2^{j(4m-M+\beta)/2} \| f_j\|   \preceq \|Kf_j\| \preceq   2^{-j(M-2m)} \left(\int_0^\infty  t^{(m-\beta)q} \left\|\cfrac{\partial^m T_{t}}{\partial t^m}f\right\|^q 
 \frac{dt}{t}\right)^{1/q} . 
\end{align}

\noindent Equivalently, 


 \begin{align}\notag
2^{j\alpha q}\| f_j\|^q   \preceq
  2^{-j/2\left(M+\beta-2\alpha q\right)}  \int_0^\infty  t^{(m-\beta)q} \left\|\cfrac{\partial^m T_{t}}{\partial t^m}f\right\|^q \frac{dt}{t} .
\end{align}
By summing over $j$ and the assumptions on $M$ and $\alpha$, 
 we arrive to the following inequality: 

 \begin{align}\notag
\sum _{j\geq 0}\left(2^{\alpha j}\| f_j\|\right)^q\preceq  \int_0^\infty  t^{(m-\beta)q} \left\|\cfrac{\partial^m T_{t}}{\partial t^m}f\right\|^q \frac{dt}{t} .
\end{align}
 Therefore

 \begin{align}\label{greater}
 \|f\|_{B_{\alpha}^q} \preceq \|f\|+  \left(\int_0^\infty  t^{(m-\beta)q} \left\|\cfrac{\partial^m T_{t}}{\partial t^m}f\right\|^q \frac{dt}{t}\right)^{1/q} .
 \end{align}
 %
%
  
  We   show that    the  converse of (\ref{greater}) is also true.  
By 
  the decomposition (\ref{resolution of identity}) and 
  an application of H\"older\rq{}s inequality for  $q$,  $p=\frac{q}{q-1}$, 
    we have
%
 \begin{align}
 \left\|\cfrac{\partial^m T_{t}}{\partial t^m}f\right\| &\leq
  \sum_{j\geq 0} \left\|\cfrac{\partial^m T_{t}}{\partial t^m}\hat\psi_j(A)f_j\right\|\\\notag
  &\leq \sum_{j\geq 0} \left\|\cfrac{\partial^m T_{t}}{\partial t^m}\hat\psi_j(A)\right\|_{op} \ \left\|f_j\right\|\\\notag
 &\leq  \left(\sum_{j\geq 0} 2^{jq\alpha}\left\|f_j\right\|^q\right)^{1/q} 
 \left(\sum_{j\geq 0} 2^{-jp\alpha}\left\|\cfrac{\partial^m T_{t}}{\partial t^m}\hat\psi_j(A)\right\|_{op}^p\right)^{1/p}   \\\notag
 &=  \|f \|_{B_q^\alpha} 
 \left(\sum_{j\geq 0} 2^{-jp\alpha}\left\|\cfrac{\partial^m T_{t}}{\partial t^m}\hat\psi_j(A)\right\|_{op}^p\right)^{1/p} 
  \end{align}
 Consequently, 
 \begin{align}\label{right integral}
 \int_0^\infty  t^{(m-\beta)q} \left\|\cfrac{\partial^m T_{t}}{\partial t^m}f\right\|^q \frac{dt}{t} &\leq
  \|f \|_{B_q^\alpha}^q
   \int_0^\infty   t^{(m-\beta)q} \left(\sum_{j\geq 0} 2^{-jp\alpha}\left\|\cfrac{\partial^m T_{t}}{\partial t^m}\hat\psi_j(A)\right\|_{op}^p \right)^{q-1} \frac{dt}{t} . 
  \end{align}

In the rest, we show that the integral on the right hand side of (\ref{right integral}) is finite and independent of $j$. We shall do this as follows.  
By the  functional calculus 
\begin{align}\notag  
\left\|\cfrac{\partial^m T_{t}}{\partial t^m}\hat\psi_j(A)\right\|_{op}= t^{-m}\sup_{\lambda}|  (\lambda t)^{m}e^{-\lambda t} \hat\psi_j(\lambda)| .
\end{align}

 Recall that   $|\hat \psi_j(\lambda)|\leq 1$ for all  $\lambda$ in the support of $\hat \psi_j$. This, along with  
 the decay property of $e^{-x}$,    gives us   the following two    estimations:  
\begin{align}\label{first}
\left\|\cfrac{\partial^m T_{t}}{\partial t^m}\hat\psi_j(A)\right\|_{op}  \preceq t^{-m} ,
\end{align}
 and  
 \begin{align}\label{second}
 \left\|\cfrac{\partial^m T_{t}}{\partial t^m}\hat\psi_j(A)\right\|_{op}\preceq 2^{m(j+1)}
 e^{-2^{j-1}t}. 
 \end{align}
 
Let $\epsilon>0$. Take $\beta\in \Bbb R$  such that   $q\beta+1>0$.
Then by  applying    (\ref{first})   
  in (\ref{right integral}), for the integral  over $(\epsilon,\infty)$  the following holds: 
\begin{align}\notag
\int_\epsilon^\infty t^{(m-\beta)q} \left(\sum_{j\geq 0} 2^{-jp\alpha}\left\|\cfrac{\partial^m T_{t}}{\partial t^m}\hat\psi_j(A)\right\|_{op}^p \right)^{q-1} \frac{dt}{t } 
    &\leq  \int_\epsilon^\infty t^{(m-\beta)q}  \left(\sum_{j\geq 0} 2^{-jp\alpha} t^{-mp} \right)^{q-1}  \frac{dt}{t} \\\notag
& =\left(\sum_{j\geq 0} 2^{-jp\alpha}   \right)^{q-1}  \int_\epsilon^\infty  t^{(m-\beta)q} t^{-mq} \cfrac{dt}{t}  \\\label{star}
&= \left(\sum_{j\geq 0} 2^{-jp\alpha}   \right)^{q-1} \int_\epsilon^\infty t^{-q\beta} \frac{dt}{t}  <\infty . 
 \end{align}

     To prove that the integral on the left   is also     finite on $(0,\epsilon)$,  we proceed  as follows. 
   %
Pick  $a$ such that     $0<a\leq m-1$. Then for any  $\lambda>1/2$ we have
       \begin{align}\label{exponential}
       t^m e^{-\lambda t} \leq   (1+(\lambda t^{-1})^a)^{-1} = \cfrac{t^a}{\lambda^a+t^a}\ ,  \quad 0<t<\epsilon .
       \end{align}

\noindent  Using the  inequality  (\ref{first}) once again, we  prove   that the    integral on the right side of  (\ref{right integral}) is finite   on $(0,\epsilon)$: 
Take $a$ sufficiently  large such that $a>\beta$ and $\alpha+a>m$.    Then  

\begin{align}\label{1}
&\int_0^\epsilon  t^{(m-\beta)q} \left(\sum_{j\geq 0} 2^{-jp\alpha}\left\|\cfrac{\partial^m T_{t}}{\partial t^m}\hat\psi_j(A)\right\|_{op}^p \right)^{q-1} \frac{dt}{t} \\\label{2} 
&\preceq
\int_0^\epsilon   t^{-\beta q} \left(\sum_{j\geq 0} 2^{-jp\alpha}   \left(\sup_{2^{j-1}\leq\lambda\leq 2^{j+1}}s^m(1+(s t^{-1})^a)^{-1}\right)^p \right)^{q-1}  \frac{dt}{t} \\\notag
&\preceq
\int_0^\epsilon   t^{-\beta q} \left(\sum_{j\geq 0} 2^{-jp\alpha}  2^{jmp} (2^{j}t^{-1})^{-ap} \right)^{q-1} \frac{dt}{t} \\\label{double-star}
&= 
   \left(\sum_{j\geq 0} 2^{-jp(\alpha+a-m)} \right)
\int_0^\epsilon   t^{(a-\beta)q} \ \frac{dt}{t}<\infty .
\end{align}

\noindent Notice that to pass from (\ref{1}) to (\ref{2}) we used   the estimations (\ref{second}) and (\ref{exponential}), respectively.  
  Consequently,  
   (\ref{double-star})  with  (\ref{star}) implies  that  the integral on right had side of (\ref{right integral})  is finite. Thus

$$\|f\|+  \left(\int_0^\infty t^{(m-\beta)q}  \left\| \cfrac{\partial^m T_t}{\partial t^m} f\right\|^q \cfrac{dt}{t} \right)^{1/q} \preceq
\|f\|_{B_q^\alpha} ,$$
 and we  have completed the proof for Theorem \ref{Stein-LPW-decomp}. 
\end{proof}

\section{applications} 
 
We conclude  this paper  by  illustrating our results in  two  examples.

 \begin{example} 
{\rm  
 The  Cauchy-Poisson  semigroup $\{P_t\}_{t\geq 0}$ on $L^2(\Bbb R^n)$, with the convention $P_0:=I$,  is given by  
 $$P_tf(x)= \int_{\RR^n} f(y) p_t(x-y) dy$$
 with the Poisson kernel 
 $$p_t(x)= c_n \cfrac{t}{(|x|^2+t^2)^{\frac{n+1}{2}}}, \quad (x,t)\in \RR^n\times (0,\infty).  $$ 
The constant $c_n$  is chosen  such that $\int p_t(x) dx =1$. The Poisson semigroup illustrates a notion of convolution semigroup and $P_tf= f\ast p_t$. 
Let $\phi^1$ denote the first derivative of the Poisson kernel for the upper half space  at time $t=1$: 
  $$\phi^{1} (x)= \cfrac{\partial}{\partial t}p_t(x)|_{t=1}~ .$$
The function  $\phi^1$ is integrable  with  mean value zero, i.e., $\int \phi^1(x) dx= 0$. 
      Similarly,  let 
  $$\phi^m(x) := \cfrac{\partial^m}{\partial t^m}p_t(x)|_{t=1}\quad \ \  m\geq 2 .$$
      It is  easy to show  that $\cfrac{\partial^m}{\partial_m t}P_tf= f\ast \phi_t^m$ where
  $\phi^m_t(x)= t^{-n}\phi^m(\frac{x}{t})$, $t>0$.  
 $\{P_t\}_{t\geq 0}$ is a symmetric diffusion semigroup $($see \cite{Stein70}$)$,  
and

 $$\|f\|_{B_q^\alpha} \asymp  \|f\| +\left(\int_0^\infty t^{(m-\alpha)q} \left\|   f\ast \phi_t^m\right\|^q \frac{dt}{t}\right)^{1/q} .$$   
 } 
 \end{example}

 \begin{example}
 {\rm 
 Let $\mathcal H=L^2(\RR^n)$. The {\it heat semigroup} $\{T_t\}_{t\geq 0}$ is defined by the Gauss-Weierstrass formula 
 \begin{align} 
 T_t f(x)=   \int_{\RR^n} f(y) h_t(x-y) dy , \quad \quad x\in \Bbb R^n, \ t>0
 \end{align}
where $h_t$ is the  heat kernel  
 $$h_t(x)= \ c_n  t^{-n/2}e^{\frac{-|x|^2}{4t}} , \quad  \quad t>0.
$$
We set $T_0:=I$. 
Here,   $c_n= \cfrac{1}{(4\pi)^{n/2}}$  and $\int_{\RR^n} h_t(x) dx=1.$
By the definition, $T_t$, $t>0$, is a convolution operator and $T_t f= f\ast h_t$. By Young\rq{}s inequality $\|T_tf\|\leq \|f\|$, and it is easy to  verify that $\{T_t\}_{t\geq 0}$ is a diffusion semigroup.  (The    semigroup axiom $T_{t+s} = T_tT_s$  can be obtained by using the Fourier transform.)  For the heat semigroup $\{T_t\}_{t\geq 0}$, the Besov norm $\|f\|_{B_q^\alpha}$  is thus  equivalent to 

$$\|f\| +\left(\int_0^\infty t^{(m-\alpha/2)q} \left\| f\ast \cfrac{\partial^m}{\partial t^m}h_t\right\|^q \frac{dt}{t}\right)^{1/q} $$
for any $m$ with $2m>\alpha$ and $\beta= \alpha/2$. 
}
 \end{example}

 \makeatletter
\renewcommand{\@biblabel}[1]{\hfill#1.}\makeatother

   \end{document}